\documentclass[reqno]{amsart}%[a4paper,10pt,reqno]{amsart}

\usepackage{mathrsfs, amsfonts, amsmath, amssymb, amscd}
\usepackage{amsbsy,bm, amsthm, enumerate}

\usepackage[english]{babel}
\usepackage{amsmath,amssymb,amsthm}
\usepackage{hyperref}
\usepackage[centering]{geometry}
\usepackage{xcolor}
\usepackage{caption}
\usepackage{booktabs}

\renewcommand{\epsilon}{\varepsilon}            %%  changes epsilon
\DeclareMathOperator{\sech}{sech}    

\newtheorem{theorem}{Theorem}[section]   %% definition of theorem environment
\newtheorem*{theorem*}{Theorem}          %% a theorem environment without numbering
\newtheorem{lemma}[theorem]{Lemma}
\newtheorem{proposition}[theorem]{Proposition}

\theoremstyle{definition}

\newtheorem{example}{Example}[section]  % 

\newtheorem{remark}{Remark}[section]
\newtheorem*{fund}{Funding}
\newtheorem*{acknow}{Acknowledgments}

\numberwithin{equation}{section}

\title[Homogeneous hypersurfaces of Sol${ }^4_0$]
{Homogeneous hypersurfaces of the four-dimensional Thurston geometry Sol${ }^4_0$}

\author{Marie D'haene\and Guoxin Wei\and Zeke Yao\and Xi Zhang}

\address{Department of Mathematics, KU Leuven,  Celestijnenlaan 200B-Box 2400, 
	3001 Leuven, Belgium}
\email{marie.dhaene@kuleuven.be}

\address{School of Mathematical Sciences, South China Normal University,
	Guangzhou 510631, People's Republic of China}
\email{weiguoxin@tsinghua.org.cn,\ yaozkleon@163.com,\ zhangxisq@163.com}

\keywords{Homogeneous hypersurface, constant principal curvature,
	solvable Lie group, Thurston geometry.}

\makeatletter
\@namedef{subjclassname@2020}{\textup{2020} Mathematics Subject Classification} 
\makeatother
\subjclass[2020]{Primary 53C42; Secondary 53C40, 53C30}

\thanks{Z. Yao is the corresponding author}

\begin{document}
	
	\begin{abstract}
		In this paper, we classify hypersurfaces with constant principal curvatures in the four-dimensional Thurston geometry ${\rm Sol_0^4}$ under certain geometric conditions. As an application of the classification result, we give a complete classification of homogeneous hypersurfaces in
		${\rm Sol_0^4}$, which solves a problem raised by Erjavec and Inoguchi (Problem 6.4 of [J. Geom. Anal. 33, Art. 274, (2023)]).
	\end{abstract}
	
	\maketitle

	%%%%%%%%%%%%%%%%%%%%%%
	\section{Introduction}\label{sect:1}
	%%%%%%%%%%%%%%%%%%%%%%
	
	A hypersurface $M$ of a Riemannian manifold $\widetilde{M}$ is called {\it homogeneous}
	if there exists a closed subgroup $G$ of the isometry group ${\rm Iso}(\widetilde{M})$
	such that $M=G \cdot p=\{g\cdot p\mid g\in G\}$ for some point $p\in M$.
	Homogeneous hypersurfaces have constant principal curvatures and
	play a fundamental role in the study of isoparametric theory.
	
	The classification of homogeneous hypersurfaces constitutes a classical problem in submanifold geometry 
	and is equivalent to the classification of cohomogeneity one actions up to orbit equivalence.
	Here, an isometric action of a connected Lie group on
	$\widetilde{M}$ is said to be of cohomogeneity one if it admits a cohomogeneity one orbit.
	Two such actions on $\widetilde{M}$ are orbit equivalent if there exists an isometry of
	$\widetilde{M}$ that maps the orbits of one action onto those of the other.
	Furthermore, cohomogeneity one actions play an essential role in constructing special geometric structures, such as Einstein metrics and metrics with special holonomies.
	
	The classification of homogeneous hypersurfaces in Euclidean space and real hyperbolic space
	$\mathbb{R}H^n$ follows from the classification of isoparametric hypersurfaces in these spaces, established by Segre \cite{S} and Cartan \cite{C}, respectively. For the unit sphere
	$\mathbb{S}^n$, Hsiang-Lawson \cite{HL} gave the classification of homogeneous hypersurfaces, while Takagi-Takahashi \cite{TT} computed their principal curvatures. Recently, Li-Ma-Wang-Wang \cite{LMWW} extended this work by classifying M\"{o}bius homogeneous hypersurfaces in $\mathbb{S}^n$, which include the classical homogeneous examples.
	In the 1970s, Takagi \cite{T1} classified homogeneous hypersurfaces in complex projective space
	$\mathbb{C}P^n$, whereas the analogous classification for complex hyperbolic space
	$\mathbb{C}H^n$ was completed much later by Berndt-Tamaru \cite{BT}. The four simply connected rank one symmetric spaces of compact type ($\mathbb{S}^n$, $\mathbb{C}P^n$, the quaternionic projective space $\mathbb{H}P^n$ and the Cayley projective plane $\mathbb{O}P^2$) have been thoroughly studied and homogeneous hypersurfaces in
	$\mathbb{H}P^n$ and $\mathbb{O}P^2$ were classified in \cite{D1,I1,I2}. More broadly, Kollross \cite{K1} classified cohomogeneity one actions on irreducible symmetric spaces of compact type up to orbit equivalence.
	For rank one symmetric spaces of non-compact type ($\mathbb{R}H^n$,
	$\mathbb{C}H^n$, the quaternionic hyperbolic space $\mathbb{H}H^n$
	and the Cayley hyperbolic plane $\mathbb{O}H^2$), homogeneous hypersurfaces in
	$\mathbb{H}H^n$ and $\mathbb{O}H^2$ were classified in \cite{BT,DDO2,DDR}. Additionally, D\'{i}az-Ramos, Dom\'{i}nguez-V\'{a}zquez and Otero \cite{DDO} developed a structural framework for cohomogeneity one actions on (not necessarily irreducible) symmetric spaces of non-compact type and arbitrary rank.
	More recently, Solonenko and Sanmart\'{i}n-L\'{o}pez \cite{SS} completed the classification of isometric cohomogeneity one actions on all symmetric spaces of non-compact type
	up to orbit equivalence.
	In reducible symmetric spaces, there are also some studies on homogeneous hypersurfaces.
	For example, Urbano \cite{U}, Gao-Ma-Yao \cite{GMY}, Dom\'{i}nguez-V\'{a}zquez and Manzano \cite{DM},
	de Lima and Pipoli \cite{LP}
	classified homogeneous hypersurfaces and isoparametric hypersurfaces of $\mathbb{S}^2\times\mathbb{S}^2$, $\mathbb{H}^2\times\mathbb{H}^2$,
	$\mathbb{S}^n\times\mathbb{R}$ and $\mathbb{H}^n\times\mathbb{R}$ ($n\geq 2$), respectively.
	Gao-Ma-Yao \cite{GMY2} gave the classification of isoparametric hypersurfaces
	in the product $M_{\kappa_1}^2\times M_{\kappa_2}^2$ of two-dimensional space forms for
	$\kappa_1, \kappa_2\in \{-1,0,1\}$ with $\kappa_1\neq\kappa_2$.
	For further developments on homogeneous hypersurfaces in symmetric spaces and advances in isoparametric theory, we refer to \cite{C1,DDO2,GQTY}.
	
	In the latter half of the last century, Thurston \cite{T} raised the geometrization conjecture: Every compact three-manifold admits a canonical decomposition into pieces modeled on one of
	eight distinct geometries. These fundamental geometries are called Thurston geometry and are given by:
	$\mathbb{R}^3$, $\mathbb{S}^3$, $\mathbb{H}^3$, $\mathbb{S}^2\times\mathbb{R}$, $\mathbb{H}^2\times\mathbb{R}$, $\widetilde{{\rm SL}}(2,\mathbb{R})$,
	${\rm Nil}^3$ and ${\rm Sol}^3$.
	Notice that $\widetilde{{\rm SL}}(2,\mathbb{R})$,
	${\rm Nil}^3$ and ${\rm Sol}^3$ are homogeneous Riemannian manifolds but not symmetric spaces, and
	the classification of homogeneous surfaces in these manifolds can be found in references \cite{DFO,DM,MP}.
	In dimension four, there are $19$ distinct families of four-dimensional Thurston geometries (cf. \cite{F,W}), which are shown in Table~\ref{tab:thurston-geometries}.
	They are ordered according to the stabilizer under the action of the identity component of their isometry groups.
	
	\begingroup
	\setlength{\tabcolsep}{10pt} % Default value: 6pt
	\renewcommand{\arraystretch}{1.8} % Default value: 1
	\begin{table}[h]
		\centering
		\renewcommand{\arraystretch}{1.3} % 
		\begin{tabular}{l l}
			\toprule
			Four-dimensional Thurston geometry    & Stabilizer \\
			\midrule
			$\mathbb{R}^4, \, \mathbb{S}^4, \, \mathbb{H}^4$ & $\text{SO}(4)$ \\ \hline
			$\mathbb{C}P^2, \, \mathbb{C}H^2$ & $\text{U}(2)$ \\ \hline
			$\mathbb{S}^3 \times \mathbb{R}, \, \mathbb{H}^3 \times \mathbb{R}$ & $\text{SO}(3)$ \\ \hline
			% &  \\ \hline
			$\mathbb{S}^2 \times \mathbb{R}^2, \, \mathbb{S}^2 \times \mathbb{S}^2, \, \mathbb{H}^2 \times \mathbb{S}^2, \, \mathbb{H}^2 \times \mathbb{R}^2, \, \mathbb{H}^2 \times \mathbb{H}^2$ & $\text{SO}(2) \times \text{SO}(2)$ \\ \hline
			$\text{Sol}_0^4, \, \widetilde{\text{SL}}(2, \mathbb{R}) \times \mathbb{R}, \, \text{Nil}^3 \times \mathbb{R}$ & $\text{SO}(2)$ \\ \hline
			${\rm F^4}$ & $(\mathbb{S}^1)_{1,2}$ \\ \hline
			$\text{Nil}^4, \, \text{Sol}_1^4, \, \text{Sol}_{m,n}^4$ & $\{1\}$ \\ 
			\bottomrule
		\end{tabular}
		\caption{Four-dimensional Thurston geometries with the stabilizer of the identity component of the associated Lie group.}
		\label{tab:thurston-geometries} 
	\end{table}
	\endgroup
	
	%\vskip -1mm
	As previously discussed, homogeneous surfaces in homogeneous three-manifolds, as well as homogeneous hypersurfaces in the first $12$ model spaces in Table~\ref{tab:thurston-geometries} have been classified.
	In this paper, we focus on the four-dimensional Thurston geometry ${\rm Sol_0^4}$ as the ambient space.
	In ${\rm Sol_0^4}$, D'haene-Inoguchi-Van der Veken \cite{DIV} classified Codazzi hypersurfaces and
	totally umbilical hypersurfaces, while Erjavec-Inoguchi \cite{EI2} studied minimal
	CR-submanifolds and raised the following fundamental problem:
	
	\vskip1mm\noindent
	{\bf Problem}. (cf. Problem 6.4 of \cite{EI2}) \ {\it Classify homogeneous hypersurfaces of ${\rm Sol_0^4}$.}
	\vskip1mm
	
	Before stating our main results, we first recall that the underlying manifold of ${\rm Sol_0^4}$ is a solvable Lie group given by \eqref{eqn:12.88}.
	The left invariant vector fields $E_1,E_2,E_3$ and $E_4$ on ${\rm Sol_0^4}$ are determined by  \eqref{eqn:2.2}. For a hypersurface $M$ of ${\rm Sol_0^4}$ with unit normal vector field $N$,
	we have four smooth functions $a, b, c, d$ on $M$ defined by
	$$
	a:=g(N,E_1), \ \ b:=g(N,E_2), \ \ c:=g(N,E_3), \ \ d:=g(N,E_4),
	$$
	where $g$ is the Riemannian metric of ${\rm Sol_0^4}$. The geometry of hypersurfaces in ${\rm Sol_0^4}$
	is closely related to these functions. In what follows, we call $a,b,c$ and $d$ the
	{\it angle functions} of $M$.
	
	Now, we study the hypersurfaces of ${\rm Sol_0^4}$ with constant principal curvatures.
	Under the assumption of constant angle functions $c$ and $d$,
	our first main result can be stated as follows.
	
	\begin{theorem}\label{thm:1.1}
		Let $M$ be a hypersurface of ${\rm Sol_0^4}$ with
		constant principal curvatures and constant angle functions $c$ and $d$.
		Then, up to isometries of ${\rm Sol_0^4}$, one of the
		following four cases occurs:
		\begin{itemize}
			
			\vskip1mm
			\item[(1)]
			$M$ is an open part of 
			\[M_{1,r}=\{(e^{x_3}\cos x_1 \tanh r,e^{x_3}\sin x_1 \tanh r,x_2,{x_3}+\ln ({\rm sech}\, r))\in {\rm Sol_0^4} \mid x_1,x_2,x_3\in \mathbb{R}\}\] 
			for some $r>0$, and has principal curvatures $\coth r$, $\tanh r$ and $-2\tanh r$ (Example~\ref{exam:3.4});
			
			\vskip1mm
			\item[(2)]
			$M$ is an open part of 
			\[M_{2,r}=\{(x_1,e^{x_3}\tanh r,x_2,{x_3}+\ln ({\rm sech}\, r))\in {\rm Sol_0^4}\mid x_1,x_2,x_3\in \mathbb{R}\}\] 
			for some $r\geq0$, and has principal curvatures $\tanh r$ (multiplicity 2) and $-2\tanh r$  (Example~\ref{exam:3.7});
			
			\vskip1mm
			\item[(3)]
			$M$ is an open part of 
			\[M_{3,r}=\left\{\left(x_1,x_2,\tfrac{1}{2}e^{-2x_3}\tanh(2r),x_3+\tfrac{1}{2}\ln(\cosh(2r))\right)\in {\rm Sol_0^4}\mid x_1,x_2,x_3\in \mathbb{R}\right\}\] 
			for some $r\geq0$, and has principal curvatures $\tanh(2r)$ (multiplicity 2) and $-2\tanh(2r)$ (Example~\ref{exam:3.8});
			
			\vskip1mm
			\item[(4)]
			$M$ is an open part of \[M_{4,0}=\{(x,y,z,0)\in {\rm Sol_0^4}\mid x,y,z\in \mathbb{R}\},\] and has principal curvatures $-2$ (multiplicity 2) and $1$ (Example~\ref{exam:3.3}).
		\end{itemize}
	\end{theorem}
	
	As a direct consequence of Theorem \ref{thm:1.1},
	we classify homogeneous hypersurfaces of ${\rm Sol_0^4}$,
	which solves the problem raised in~\cite{EI2}.
	
	\begin{theorem}\label{thm:1.2}
		Let $M$ be a homogeneous hypersurface of ${\rm Sol_0^4}$.
		Then, up to isometries of ${\rm Sol_0^4}$, one of the
		following four cases occurs:
		\begin{itemize}
			
			\vskip1mm
			\item[(1)]
			$M$ is $M_{1,r}$ for some $r>0$, which is an orbit through $(\tanh r,0,0,\ln(\sech r))$ of the subgroup 
			$\{(0,0,z,t)\in {\rm Sol_0^4}\mid z,t\in \mathbb{R}\}\times {\rm SO(2)} \cong (\mathbb R \rtimes_{\theta_1} \mathbb R) \times \rm{SO}(2)$ where $\theta_1(t) = e^{-2t}$ (Example~\ref{exam:3.4});
			
			\vskip1mm
			\item[(2)]
			$M$ is $M_{2,r}$ for some $r\geq0$, which is an orbit through $(0,\tanh r,0,\ln(\sech r))$ of the subgroup 
			$\{(x,0,z,t)\in\mathrm{Sol}^4_0 \mid x,z,t \in \mathbb R \} \cong \mathbb R^2 \rtimes_{\theta_2} \mathbb R$ 
			where $\theta_2(t) = \left(\begin{smallmatrix} e^{t} & 0 \\ 0 & e^{-2t} \end{smallmatrix}\right)$ (Example~\ref{exam:3.7});
			
			\vskip1mm
			\item[(3)]
			$M$ is $M_{3,r}$ for some $r\geq0$, which is an orbit through $(0,0,\tfrac12 \tanh(2r),\tfrac12 \ln(\cosh(2r)))$ 
			of the subgroup 
			$\{(x,y,0,t)\in\mathrm{Sol}^4_0 \mid x,y,t \in \mathbb R\} \cong \mathbb R^2 \rtimes_{\theta_3} \mathbb R$ where $\theta_3(t) = \left(\begin{smallmatrix} e^{t} & 0 \\ 0 & e^{t} \end{smallmatrix}\right)$ (Example~\ref{exam:3.8});
			
			\vskip1mm
			\item[(4)]
			$M$ is $M_{4,0}$, which is an orbit through $(0,0,0,0)$ of the subgroup
			$\{(x,y,z,0) \in \mathrm{Sol}^4_0 \mid x,y,z \in \mathbb R\} \cong \mathbb R^3$ (Example~\ref{exam:3.3}).
		\end{itemize}
	\end{theorem}

	\begin{remark}\label{rm:1.1}
		It is noteworthy that the hypersurfaces $\{M_{2,r},\ r\geq0\}$
		(resp. $\{M_{3,r}, r\geq0\}$) constitute a family of parallel hypersurfaces, which are always
		minimal, but for any given $r_1$ and $r_2$, the hypersurfaces $M_{2,r_1}$ and $M_{2,r_2}$
		(resp. $M_{3,r_1}$ and $M_{3,r_2}$) are not congruent to each other.
	\end{remark}
	
	This paper is organized as follows: In Section \ref{sect:2}, we collect some necessary
	materials about the ambient space ${\rm Sol_0^4}$ as well as its hypersurfaces.
	In Section \ref{sect:3}, we describe in detail the examples characterized by 
	Theorems \ref{thm:1.1} and \ref{thm:1.2}.
	Finally, we prove Theorems \ref{thm:1.1} and \ref{thm:1.2} in Section \ref{sect:4}.

	\section{Preliminaries}\label{sect:2}
	%=================================================
	%%%%%%%%%%%%%%%%%%%%%%%%%%%%%%%%%%%%%%%%%%%%%
	\subsection{The geometric structure on \texorpdfstring{${\rm Sol_0^4}$}{Sol40}}\label{sect:2.1}~
	%%%%%%%%%%%%%%%%%%%%%%%%%%%%%%%%%%%%%%%%%%%%%
	
	In this subsection, we review some basic materials from \cite{DIV,EI2}.
	The underlying manifold of ${\rm Sol_0^4}$ is the following
	connected solvable Lie group
	\begin{equation}\label{eqn:12.88}
		\left\{
		(x, y, z, t):=
		\left.\left(
		\begin{matrix}
			e^t & 0   & 0       & x \\
			0   & e^t & 0       & y \\
			0   & 0   & e^{-2t} & z \\
			0   & 0   & 0       & 1 \\
		\end{matrix}
		\right) \right| \ x,y,z,t\in \mathbb{R}
		\right\}.
	\end{equation}
	Equivalently, the underlying manifold is
	$\mathbb{R}^4=\{(x,y,z,t)|\, x,y,z,t\in \mathbb{R}\}$ with group operation
	\begin{equation}\label{eqn:2.1}
		(\tilde{x},\tilde{y},\tilde{z},\tilde{t})\cdot(x,y,z,t)
		=(\tilde{x}+e^{\tilde{t}}x,\tilde{y}+e^{\tilde{t}}y,\tilde{z}+e^{-2\tilde{t}}z,\tilde{t}+t).
	\end{equation}
	% Here, we use $(x,y,z,t)$ to denote Cartesian coordinates on $\mathbb{R}^4$.
	% The inverse of $(x,y,z,t)$ is given by
	% $$
	% (x,y,z,t)^{-1}=(-e^{-t}x,-e^{-t}y,-e^{2t}z,-t).
	% $$
	
	At a point $p=(x,y,z,t)\in {\rm Sol_0^4}$, we define four left invariant vector fields by
	\begin{equation}\label{eqn:2.2}
		E_1=e^t \partial_x, \quad
		E_2=e^t \partial_y,\quad
		E_3=e^{-2t} \partial_z, \quad
		E_4= \partial_t.
	\end{equation}
	The following commutation relations hold:
	\begin{equation}\label{eqn:2.3}
		\begin{aligned}
			&[E_1,E_2]=[E_1,E_3]=[E_2,E_3]=0,\\
			&[E_1,E_4]=-E_1, \ \ [E_2,E_4]=-E_2, \ \ [E_3,E_4]=2E_3.
		\end{aligned}
	\end{equation}

	The Lie group associated to the Thurston geometry $\mathrm{Sol}^4_0$ is the semi-direct product $\mathrm{Sol}^4_0 \rtimes (\mathrm{O}(2) \times \mathbb Z_2)$, where ${\rm Sol_0^4}$ acts on itself by left translations, ${\rm O(2)}$ acts by rotations and reflections in the $xy$-plane and $\mathbb{Z}_2$ acts by reflections in the $z$-coordinate.
	The identity component of this group is $\mathrm{Sol}^4_0 \rtimes \mathrm{SO}(2)$, which is isomorphic to the matrix group
	\begin{equation}\label{eqn:2.100}
		\left\{\left.\left(
		\begin{matrix}
			e^t \cos \theta & -e^t \sin \theta & 0 & x \\
			e^t \sin \theta & e^t \cos \theta & 0 & y \\
			0 & 0 & e^{-2t} & z \\
			0 & 0 & 0 & 1
		\end{matrix}
		\right) \right|x, \, y, \, z, \, t \in \mathbb{R}, \, 0\leq\theta< 2\pi\right\}.
	\end{equation}
	
	One can show that there exists a unique Riemannian metric $g$ on ${\rm Sol_0^4}$ that is invariant under the action of     $\mathrm{Sol}^4_0 \rtimes (\mathrm{O}(2) \times \mathbb Z_2)$ (see~\cite{DIV}).
	It is given by
	\[
	g=e^{-2t}(dx^2+dy^2)+e^{4t}dz^2+dt^2.
	\]
	Note that the frame field $\{E_i\}_{i=1}^4$ in~\eqref{eqn:2.2} is orthonormal with respect to $g$.
	
	%%%%%%%%%
	% Since a Thurston geometry is a pair of a manifold X and a Lie group G (satisfying certain conditions), we should introduce the associated Lie group first, and then mention all the Riemannian metrics which are invariant under the action of this Lie group. In the case of Sol^4_0, there is a unique metric satisfying this property. See "Parallel and totally umbilical hypersurfaces in the four-dimensional Thurston geometry Sol^4_0" for a detailed explanation.
	%%%%%%%%%
	% According to \cite{D,DIV,W}, the isometry group ${\rm Iso(Sol_0^4)}$ of ${\rm Sol_0^4}$ is
	% the semi-direct product ${\rm Sol_0^4}\rtimes({\rm O(2)}\times\mathbb{Z}_2)$,
	% where ${\rm Sol_0^4}$ acts on itself by left
	% translations, ${\rm O(2)}$ acts by rotations and
	% reflections in the $xy$-plane and $\mathbb{Z}_2$
	% acts by reflections in the $z$-coordinate. Moreover,
	% the identity component of the isometry group of ${\rm Sol_0^4}$ is given by
	% %
	% \begin{equation}\label{eqn:2.100}
		% 	{\rm Sol_0^4}\rtimes{\rm SO(2)}=\left\{\left.\left(
		% 	\begin{matrix}
			% 		e^t \cos \theta & -e^t \sin \theta & 0 & x \\
			% 		e^t \sin \theta & e^t \cos \theta & 0 & y \\
			% 		0 & 0 & e^{-2t} & z \\
			% 		0 & 0 & 0 & 1
			% 	\end{matrix}
		% 	\right) \right|x, \, y, \, z, \, t \in \mathbb{R}, \, 0\leq\theta< 2\pi\right\}.
		% \end{equation}
	
	Let $\tilde{\nabla}$ be the Levi-Civita connection associated to $g$.
	By using \eqref{eqn:2.3} and Koszul's formula, we obtain
	\begin{equation}\label{eqn:2.6}
		\begin{alignedat}{4}
			&\tilde{\nabla}_{E_1}E_1=E_4, \qquad 
			&&\tilde{\nabla}_{E_1}E_2=0, \qquad
			&&\tilde{\nabla}_{E_1}E_3=0, \qquad
			&&\tilde{\nabla}_{E_1}E_4=-E_1, \\
			&\tilde{\nabla}_{E_2}E_1=0, \qquad 
			&&\tilde{\nabla}_{E_2}E_2=E_4\qquad \
			&&\tilde{\nabla}_{E_2}E_3=0, \qquad
			&&\tilde{\nabla}_{E_2}E_4=-E_2, \\
			&\tilde{\nabla}_{E_3}E_1=0, \qquad 
			&&\tilde{\nabla}_{E_3}E_2=0, \qquad
			&&\tilde{\nabla}_{E_3}E_3=-2E_4, \qquad \ \
			&&\tilde{\nabla}_{E_3}E_4=2E_3, \\
			&\tilde{\nabla}_{E_4}E_1=0, \qquad 
			&&\tilde{\nabla}_{E_4}E_2=0, \qquad
			&&\tilde{\nabla}_{E_4}E_3=0, \qquad
			&&\tilde{\nabla}_{E_4}E_4=0.
		\end{alignedat}
	\end{equation}

	On ${\rm Sol_0^4}$, we may define two complex structures, $J_1$ and $J_2$, by
	\begin{equation}\label{eqn:2.5}
		\begin{alignedat}{4}
			&J_1E_1=E_2, \quad &&J_1E_2=-E_1, \quad
			&&J_1E_3=E_4, \quad &&J_1E_4=-E_3, \\
			&J_2E_1=E_2, \quad &&J_2E_2=-E_1, \quad
			&&J_2E_3=-E_4, \quad &&J_2E_4=E_3.
		\end{alignedat}
	\end{equation}
	Both structures are compatible and integrable with respect to $g$.
	Note that none of these structures is K\"ahler.
	Additionally, we define another relevant tensor $P$ on ${\rm Sol_0^4}$ in the following way:
	\begin{equation}\label{eqn:2.49}
		PX=g(X,E_4)E_4, \qquad \text{ where }  X\in \mathfrak X({\rm Sol_0^4}).
	\end{equation}
	Using these structures, we may express the Riemannian curvature tensor $\tilde{R}$ of $({\rm Sol_0^4},g)$
	in the following way:
	\begin{equation}\label{eqn:2.7}
		\begin{aligned}
			\tilde{R}(X,Y)Z=\ 
			&2\big(g(Y,Z)X-g(X,Z)Y\big)\\
			&-\tfrac{1}{2}\big(g(J_1Y,Z)J_1X-g(J_1X,Z)J_1Y+2g(X,J_1Y)J_1Z\big)\\
			&-\tfrac{1}{2}\big(g(J_2Y,Z)J_2X-g(J_2X,Z)J_2Y+2g(X,J_2Y)J_2Z\big)\\
			&-3\big(g(PY,Z)X+g(Y,Z)PX-g(PX,Z)Y-g(X,Z)PY\big),
		\end{aligned}
	\end{equation}
	where $X, Y, Z\in \mathfrak X({\rm Sol_0^4})$.

	\subsection{Hypersurfaces of \texorpdfstring{${\rm Sol_0^4}$}{Sol40}}\label{sect:2.2}~
	%%%%%%%%%%%%%%%%%%%%%%%%%%%%%%%%%%%%%%%%%%%%%
	
	Let $M$ be an isometrically immersed hypersurface of ${\rm Sol_0^4}$ with unit normal vector field $N$.
	We may write
	$$
	N=aE_1+bE_2+cE_3+dE_4,
	$$
	where $\{E_i\}_{i=1}^4$ is defined by \eqref{eqn:2.2}, and
	$a, b, c, d$ are smooth functions on $M$, called angle functions, satisfying
	$a^2+b^2+c^2+d^2=1$.
	Then the following vector fields form an orthonormal frame on $M$:
	\begin{equation}\label{eqn:8.1}
		\begin{aligned}
			&T_1=bE_1-aE_2+dE_3-cE_4, \\
			&T_2=cE_1-dE_2-aE_3+bE_4, \\
			&T_3=dE_1+cE_2-bE_3-aE_4. \\
		\end{aligned}
	\end{equation}
	
	Let $\nabla$ be the induced connection on $M$. The Gauss and Weingarten formulae, respectively, are
	$$
	\tilde{\nabla}_XY=\nabla_XY+g(AX,Y)N, \qquad \tilde{\nabla}_XN=-AX, 
	$$
	where $X,Y \in \mathfrak X(M)$ and $A$ is the shape operator of $M$.
	
	Let $R$ denote the Riemannian curvature tensor of $M$. The Gauss and Codazzi equations
	of $M$ are given by
	\begin{align}
		(\tilde R(X,Y)Z)^\top &=  R(X,Y)Z - g(AY,Z)AX + g(AX,Z)AY, \label{eqn:2.14} \\
		(\tilde R(X,Y)Z)^\perp &= g((\nabla_X A)Y - (\nabla_Y A)X,Z) N, \label{eqn:2.15}
	\end{align}
	%
	%       \begin{equation}
		%       \begin{aligned}
			% 	R(X,Y)Z=&2\Big(g(Y,Z)X-g(X,Z)Y\Big)\\
			% 	&-\tfrac{1}{2}\Big(g(J_1Y,Z)(J_1X)^\top-g(J_1X,Z)(J_1Y)^\top+2g(X,J_1Y)(J_1Z)^\top\Big)\\
			% 	&-\tfrac{1}{2}\Big(g(J_2Y,Z)(J_2X)^\top-g(J_2X,Z)(J_2Y)^\top+2g(X,J_2Y)(J_2Z)^\top\Big)\\
			% 	&-3\Big(g(PY,Z)X+g(Y,Z)(PX)^\top-g(PX,Z)Y-g(X,Z)(PY)^\top\Big)\\
			% 	&+g(AY,Z)AX-g(AX,Z)AY,
			% \end{aligned}
		%       \end{equation}
	%       \begin{equation}
		%           \begin{aligned}
			% 	g((\nabla_XA)&Y-(\nabla_YA)X,Z)\\
			% 	&=-\tfrac{1}{2}\Big(g(J_1Y,Z)g(J_1X,N)-g(J_1X,Z)g(J_1Y,N)+2g(X,J_1Y)g(J_1Z,N)\Big)\\
			% 	& \ \ \ -\tfrac{1}{2}\Big(g(J_2Y,Z)g(J_2X,N)-g(J_2X,Z)g(J_2Y,N)+2g(X,J_2Y)g(J_2Z,N)\Big)\\
			% 	& \ \ \ -3\Big(g(Y,Z)g(PX,N)-g(X,Z)g(PY,N)\Big),
			% \end{aligned}
		%       \end{equation}
	where $X,Y,Z\in \mathfrak X(M)$, and $\cdot^\top$ and $\cdot^\perp$ denote the tangential and normal components, respectively. 
	It follows from \eqref{eqn:2.14} that the Ricci curvature tensor of $M$ is given by
	\begin{equation}\label{eqn:MRic}
		\begin{aligned}
			g({\rm Ric}(X),Y)=&\left(-2+3d^2\right)g(X,Y)+\tfrac{3}{2}g(X,J_1N)g(Y,J_1N)+\tfrac{3}{2}g(X,J_2N)g(Y,J_2N)\\
			&-3g(PX,Y)+Hg(AX,Y)-g(A^2X,Y),
		\end{aligned}
	\end{equation}
	where $H={\rm Tr}A$ denotes the mean curvature of $M$.

	%%%%%%%%%%%%%%%%%%%%%%%%%%%%%%%%%%%%%%%%%%%%%
	\section{Examples of homogeneous hypersurfaces of \texorpdfstring{${\rm Sol_0^4}$}{Sol40}}\label{sect:3}

	In this section, we introduce some examples of hypersurfaces in ${\rm Sol_0^4}$
	which are characterized by Theorems \ref{thm:1.1} and \ref{thm:1.2}.
	\begin{example}\label{exam:3.4}
		For any given $r>0$, we define the hypersurface
		$$
		M_{1,r}:=\{(e^{x_3}\cos x_1 \tanh r,e^{x_3}\sin x_1 \tanh r,x_2,{x_3}+\ln ({\rm sech}\, r))\in {\rm Sol_0^4}
		\mid x_1,x_2,x_3\in \mathbb{R}\}.
		$$
		Let $H_1:=\{(0,0,z,t)\in {\rm Sol_0^4}\mid z,t\in \mathbb{R}\}\times {\rm SO(2)}$.
		\begin{proposition}\label{prop:M1r}
			As a hypersurface of ${\rm Sol_0^4}$,
			$M_{1,r}$ has the following properties:
			\begin{itemize}
				
				\vskip1mm
				\item[(1)]
				The hypersurface $M_{1,r}$ can also be presented as $\{(x,y,z,t)\in {\rm Sol_0^4}
				\mid  (x^2+y^2)e^{-2t}=\sinh^2 r\}$. 
				The unit normal vector field is given by
				\begin{equation}\label{eqn:8.16}
					N=-\cos {x_1}\, {\rm sech}\, r E_1-\sin {x_1}\, {\rm sech}\, r E_2+\tanh r E_4,
				\end{equation}
				where $\{E_i\}_{i=1}^4$ is defined by \eqref{eqn:2.2};
				
				\vskip1mm
				\item[(2)]
				It has three distinct constant principal curvatures $\coth r$, $\tanh r$ and $-2\tanh r$. The eigenvalues of the Ricci tensor of $M_{1,r}$ are $0$ and $-4 {\rm sech}^2 r$ (multiplicity $2$); 
				
				\vskip1mm
				\item[(3)]
				It is an orbit of the subgroup $H_1$ which passes through the point
				$(\tanh r,0,0,\ln({\rm sech}\, r))$.
				Therefore $M_{1,r}$ is a homogeneous hypersurface;
				
				\vskip1mm
				\item[(4)]
				It is a tube of radius $r>0$ over the focal manifold
				$\{(0,0,z,t)\in {\rm Sol_0^4}\mid z,t\in \mathbb{R}\}$.
			\end{itemize}
		\end{proposition}
		
		\begin{proof}
			(1) At any $p\in M_{1,r}$, we consider the frame field
			$\{V_i\}_{i=1}^3$ given by
			\begin{equation*}
				\begin{aligned}
					&V_1=p_{x_1}=-\sin {x_1}\sinh r E_1+\cos {x_1}\sinh r E_2,\\
					&V_2=p_{x_3}=\cos {x_1}\sinh r E_1+\sin {x_1}\sinh r E_2+E_4,\\
					&V_3=p_{x_2}=e^{2 x_3}{\rm sech^2}\, r E_3,
				\end{aligned}
			\end{equation*}
			so $V_1(x_1)=1$.
			Hence,~\eqref{eqn:8.16} is the unit normal vector field.
			Additionally, a direct computation shows that $M_{1,r}$ equals $\{(x,y,z,t)\in {\rm Sol_0^4}
			\mid  (x^2+y^2)e^{-2t}=\sinh^2 r\}$.
			
			(2) Consider the orthonormal frame field $\{ W_i \}_{i=1}^3$ on $M_{1,r}$ given by:
			\begin{equation}\label{eqn:8.17}
				\begin{aligned}
					&W_1=\frac{1}{\sinh r}V_1=-\sin x_1E_1+\cos x_1E_2,\\
					&W_2=\frac{1}{\cosh r}V_2=\cos x_1\tanh r E_1+\sin x_1\tanh r E_2+{\rm sech}\, rE_4,\\
					&W_3=\frac{1}{e^{2x_3}{\rm sech^2}\, r}V_3=E_3.
				\end{aligned}
			\end{equation}
			Then, $W_1(-\cos x_1\, {\rm sech}\, r)=\tfrac{1}{\sinh r}V_1(-\cos x_1\, {\rm sech}\, r)
			=2\sin x_1\, {\rm csch}\, (2r)$, which does not vanish identically.
			
			By direct calculations using \eqref{eqn:2.6}, \eqref{eqn:8.16},
			\eqref{eqn:8.17} and the Weingarten formula, we obtain
			\begin{equation}\label{eqn:M1rz}
				AW_1=\coth r W_1, \quad AW_2=\tanh r W_2, \quad AW_3=-2\tanh r W_3.
			\end{equation}
			Then, by \eqref{eqn:2.14} and \eqref{eqn:M1rz}, it follows that the sectional curvature of $M_{1,r}$ satisfies
			\begin{equation}\label{eqn:M1rsec}
				K(W_1\wedge W_2)=K(W_1\wedge W_3)=0, \quad K(W_2\wedge W_3)=-4 {\rm sech}^2 r.
			\end{equation}
			The Ricci curvature of $M_{1,r}$ is given by 
			\begin{equation}\label{eqn:M1rric}
				{\rm Ric}(W_1)=0, \quad {\rm Ric}(W_2)=-4 {\rm sech}^2 r W_2,\quad {\rm Ric}(W_3)=-4 {\rm sech}^2 r W_3. 
			\end{equation}
			
			(3)-(4) Using the presentation~\eqref{eqn:2.100} of $\mathrm{Sol}^4_0 \rtimes \mathrm{SO}(2)$,
			we can directly verify that $H_1$ is a closed subgroup of $\mathrm{Sol}^4_0 \rtimes \mathrm{SO}(2)$,
			%              we define
			% $$
			% g_i=\left(
			% \begin{matrix}
				% 	e^{t_i} \cos \theta_i & -e^{t_i} \sin \theta_i & 0 & 0 \\
				% 	e^{t_i} \sin \theta_i & e^{t_i} \cos \theta_i & 0 & 0 \\
				% 	0 & 0 & e^{-2t_i} & z_i \\
				% 	0 & 0 & 0 & 1
				% \end{matrix}
			% \right)\in H_1,
			% $$
			%              where $z_i,t_i\in \mathbb{R}$ and $\theta_i\in [0,2\pi)$, for $i=1,2$.
			% %
			% Then we have
			% \begin{equation*}
				% 	g_1g_2=\left(
				% 	\begin{matrix}
					% 		e^{t_1+t_2} \cos (\theta_1+\theta_2) & -e^{t_1+t_2} \sin (\theta_1+\theta_2) & 0 & 0 \\
					% 		e^{t_1+t_2} \sin (\theta_1+\theta_2) & e^{t_1+t_2} \cos (\theta_1+\theta_2) & 0 & 0 \\
					% 		0 & 0 & e^{-2(t_1+t_2)} & z_1+e^{-2t_1}z_2 \\
					% 		0 & 0 & 0 & 1
					% 	\end{matrix}
				% 	\right)\\
				% 	\in H_1.
				% \end{equation*}
			% Thus $H_1$ is a closed subgroup of ${\rm Sol_0^4}\rtimes{\rm SO(2)}$.
			% Moreover, one can verify that
			and that the hypersurface $M_{1,r}$ is an orbit
			of the subgroup $H_1$ which passes through $(\tanh r,0,0,\ln({\rm sech}\, r))$.
			The action of $H_1$ on ${\rm Sol_0^4}$ has only one singular orbit,
			$\{(0,0,z,t)\in {\rm Sol_0^4}\mid z,t\in \mathbb{R}\}$, which passes through the origin $(0,0,0,0)$.
			It can be checked that $M_{1,r}$ is a tube of radius $r$
			over the focal manifold
			$\{(0,0,z,t)\in {\rm Sol_0^4}\mid z,t\in \mathbb{R}\}$.
		\end{proof}
	\end{example}
	
	\begin{example}\label{exam:3.7}
		For any given $r\geq0$, we define the hypersurface
		$$
		M_{2,r}:=\{(x_1,e^{x_3}\tanh r,x_2,{x_3}+\ln ({\rm sech}\, r))\in {\rm Sol_0^4}
		\mid x_1,x_2,x_3\in \mathbb{R}\}.
		$$
		Let $H_2:=\{(x,0,z,t)\in {\rm Sol_0^4}\mid x,z,t\in \mathbb{R}\}$.
		\begin{proposition}\label{prop:M2r}
			As a hypersurface of ${\rm Sol_0^4}$,
			$M_{2,r}$ has the following properties:
			\begin{itemize}
				
				\vskip1mm
				\item[(1)]
				The hypersurface $M_{2,r}$ can also be presented as $\{(x,y,z,t)\in {\rm Sol_0^4}
				\mid  ye^{-t}=\sinh r\}$.  
				The unit normal vector field is given by $N=-{\rm sech}\, r E_2+\tanh r E_4$,
				where $\{E_i\}_{i=1}^4$ is defined by~\eqref{eqn:2.2};
				
				\vskip1mm
				\item[(2)]
				It is a minimal hypersurface and has constant principal curvatures
				$\tanh r$ (multiplicity $2$) and $-2\tanh r$. 
				The eigenvalues of the Ricci tensor of $M_{2,r}$ are ${\rm sech}^2 r$, $-5 {\rm sech}^2 r$ and $-2 {\rm sech}^2 r$.
				When $r=0$, $M_{2,0}$ is totally geodesic;
				
				\vskip1mm
				\item[(3)]
				It is an orbit of the subgroup $H_2$
				which passes through $(0,\tanh r,0,\ln ({\rm sech}\, r))$. Therefore,
				$M_{2,r}$ is a homogeneous hypersurface, and it has no focal manifold.
			\end{itemize}
		\end{proposition}
		
		\begin{proof}
			(1) At any $p\in M_{2,r}$, we consider the frame field
			$\{V_i\}_{i=1}^3$ given by
			$$
			V_1=p_{x_1}=e^{-x_3}\cosh r E_1,
			\quad V_2=p_{x_3}=\sinh r E_2+E_4,
			\quad V_3=p_{x_2}=e^{2x_3}{\rm sech^2}\, r E_3.
			$$
			Hence, $N=-{\rm sech}\, r E_2+\tanh r E_4$ is the unit normal vector field.
			Additionally, a direct computation shows that $M_{2,r}$ equals $\{(x,y,z,t)\in {\rm Sol_0^4}
			\mid  ye^{-t}=\sinh r\}$.
			
			(2) Consider the orthonormal frame field $\{W_i\}_{i=1}^3$ on $M_{2,r}$ given by:
			$$
			W_1=\frac{V_1}{e^{-x_3}\cosh r}=E_1, \quad W_2=\frac{V_2}{\cosh r}=\tanh r E_2+{\rm sech}\, r E_4, \quad
			W_3=\frac{V_3}{e^{2x_3}{\rm sech^2}\, r}=E_3.
			$$
			By direct calculations using~\eqref{eqn:2.6} and the Weingarten formula, we obtain
			\begin{equation}\label{eqn:M2rz}
				AW_1=\tanh r W_1, \quad AW_2=\tanh r W_2, \quad AW_3=-2\tanh r W_3.
			\end{equation}
			Then, by \eqref{eqn:2.14} and \eqref{eqn:M2rz}, it follows that  the sectional curvature of $M_{2,r}$ satisfies
			\begin{equation}\label{eqn:M2rsec}
				K(W_1\wedge W_2)=-{\rm sech}^2 r,\quad K(W_1\wedge W_3)=2 {\rm sech}^2 r, \quad K(W_2\wedge W_3)=-4 {\rm sech}^2 r.
			\end{equation}
			The Ricci curvature of $M_{2,r}$ is given by 
			\begin{equation}\label{eqn:M2rric}
				{\rm Ric}(W_1)={\rm sech}^2 r W_1, \quad {\rm Ric}(W_2)=-5 {\rm sech}^2 r W_2,\quad {\rm Ric}(W_3)=-2 {\rm sech}^2 r W_3. 
			\end{equation}

			(3) Using~\eqref{eqn:2.1},
			we can directly verify that $H_2$ is a closed subgroup of
			${\rm Sol_0^4 \rtimes SO(2)}$.
			Moreover, it follows from~\eqref{eqn:2.1} that
			the hypersurface $M_{2,r}$ is an orbit of the subgroup $H_2$
			which passes through $(0,\tanh r,0,\ln ({\rm sech}\, r))$ and
			it has no focal manifold.
		\end{proof}
	\end{example}
	
	\begin{example}\label{exam:3.8}
		For any given $r\geq0$, we define the hypersurface
		$$
		M_{3,r}:=\left\{\left(x_1,x_2,\tfrac{1}{2}e^{-2x_3}\tanh(2r),x_3+\tfrac{1}{2}\ln(\cosh(2r))\right)\in {\rm Sol_0^4}
		\mid x_1,x_2,x_3\in \mathbb{R}\right\}.
		$$
		Let $H_3:=\{(x,y,0,t)\in {\rm Sol_0^4}\mid x,y,t\in \mathbb{R}\}$.
		\begin{proposition}\label{prop:M3r}
			As a hypersurface of ${\rm Sol_0^4}$, $M_{3,r}$ has the following properties:
			\begin{itemize}
				
				\vskip1mm
				\item[(1)]
				The hypersurface $M_{3,r}$ can also be presented as $\{(x,y,z,t)\in {\rm Sol_0^4}
				\mid  2ze^{2t}=\sinh (2r)\}$.  
				The unit normal vector field is given by $N={\rm sech}(2r)E_3+\tanh(2r)E_4$,
				where $\{E_i\}_{i=1}^4$ is defined by \eqref{eqn:2.2};
				
				\vskip1mm
				\item[(2)]
				It is a minimal hypersurface and has constant principal curvatures $\tanh(2r)$ (multiplicity $2$)
				and $-2\tanh(2r)$.
				The hypersurface $M_{3,r}$ has constant sectional curvature $-{\rm sech}^2 (2r)$.
				When $r=0$, $M_{3,0}$ is a totally geodesic hypersurface;
				
				\vskip1mm
				\item[(3)]
				It is an orbit of the subgroup $H_3$ which passes through
				$(0,0,\tfrac{1}{2}\tanh(2r),\tfrac{1}{2}\ln(\cosh(2r)))$.
				Therefore, $M_{3,r}$ is a homogeneous hypersurface,
				and it has no focal manifold.
			\end{itemize}
		\end{proposition}
		
		\begin{proof}
			(1) At any $p\in M_{3,r}$, we consider the frame field $\{V_i\}_{i=1}^3$ given by
			\begin{equation*}
				\begin{aligned}
					V_1&=p_{x_1}=e^{-\left(x_3+\tfrac{1}{2}\ln(\cosh(2r))\right)}E_1,\\
					V_2&=p_{x_2}=e^{-\left(x_3+\tfrac{1}{2}\ln(\cosh(2r))\right)}E_2,\\
					V_3&=p_{x_3}=-\sinh(2r)E_3+E_4.
				\end{aligned}
			\end{equation*}
			Hence, $N={\rm sech}(2r)E_3+\tanh(2r)E_4$ is the unit normal vector field.
			Additionally, a direct computation shows that $M_{3,r}$ equals $\{(x,y,z,t)\in {\rm Sol_0^4}
			\mid  2ze^{2t}=\sinh (2r)\}$.
			
			(2) Consider the orthonormal frame field $\{W_i\}_{i=1}^3$ on $M_{3,r}$ given by:
			$$
			W_1=E_1,
			\quad W_2=E_2, \quad
			W_3=\frac{V_3}{\cosh(2r)}=-\tanh (2r)E_3+{\rm sech}(2r)E_4.
			$$
			By direct calculations using~\eqref{eqn:2.6} and the Weingarten formula, we obtain
			\begin{equation}\label{eqn:M3rz}
				AW_1=\tanh(2r)W_1, \quad AW_2=\tanh(2r)W_2, \quad AW_3=-2\tanh(2r)W_3.
			\end{equation}
			Then, by \eqref{eqn:2.14} and \eqref{eqn:M3rz}, it follows that the sectional curvature of $M_{3,r}$ satisfies
			\begin{equation}\label{eqn:M3rsec}
				K(W_1\wedge W_2)=K(W_1\wedge W_3)=K(W_2\wedge W_3)=-{\rm sech}^2 (2r).
			\end{equation}
			Since ${\rm dim}\ M_{3,r}=3$, $M_{3,r}$ has constant sectional curvature $-{\rm sech}^2 (2r)$.

			(3) Using~\eqref{eqn:2.1},
			we can directly verify that $H_3$ is a closed subgroup of
			${\rm Sol_0^4 \rtimes SO(2)}$. Moreover, it follows from \eqref{eqn:2.1} that
			the hypersurface $M_{3,r}$ is an orbit of the subgroup $H_3$
			which passes through $(0,0,\tfrac{1}{2}\tanh(2r),\tfrac{1}{2}\ln(\cosh(2r)))$ and
			it has no focal manifold.
		\end{proof}
	\end{example}
	
	\begin{remark}\label{rm:1.3}
		The totally geodesic hypersurfaces of ${\rm Sol_0^4}$ have
		been classified in~\cite[Cor.3.4]{DIV}.
		We point out that, after rotations in the $xy$-plane and left translations,
		the hypersurfaces $M_{3,0}$ and $M_{2,0}$ are congruent to the examples $(a)$ and $(b)$
		of Corollary 3.4, respectively.
	\end{remark}
	
	\begin{example}[cf.\ \cite{DIV,EI2}]\label{exam:3.3}
		For any given $r\in \mathbb{R}$, we consider the hypersurface $M_{4,r}$ defined by
		$$
		M_{4,r}:=\{(x,y,z,r)\in {\rm Sol_0^4}\mid x,y,z\in \mathbb{R}\}.
		$$
		
		Let $H_4:=\{(x,y,z,0)\in{\rm Sol_0^4}\mid x,y,z\in \mathbb{R}\}$.
		Then by \eqref{eqn:2.1}, we can directly verify that $H_4$ is a closed subgroup of ${\rm Sol_0^4 \rtimes SO(2)}$.
		By using \eqref{eqn:2.1}, \eqref{eqn:2.2}, \eqref{eqn:2.6} and the Weingarten formula,
		we have the following proposition without proof.
		\begin{proposition}\label{prop:M4r}
			As a hypersurface of ${\rm Sol_0^4}$,
			$M_{4,r}$ has the following properties:
			\begin{itemize}
				
				\vskip1mm
				\item[(1)]
				The unit normal vector field is given by $N=E_4$, where $E_4$ is defined by \eqref{eqn:2.2};
				
				\vskip1mm
				\item[(2)]
				It is a minimal hypersurface and has two distinct constant principal
				curvatures $1$ (multiplicity $2$) and $-2$. 
				The hypersurface $M_{4,r}$ has constant sectional curvature $0$;
				
				\vskip1mm
				\item[(3)]
				It is an orbit of the subgroup $H_4$ which passes through
				$(0,0,0,r)$. Therefore, the hypersurface $M_{4,r}$ is a homogeneous hypersurface,
				and it has no focal manifold. Moreover, for any given~$r$,
				the hypersurface $M_{4,r}$ is congruent to $M_{4,0}$.
			\end{itemize}
		\end{proposition}
	\end{example}
	
	%%%%%%%%%%%%%%%%%%%%%%%%%%%%%%%%%%%%%%%%%%%%%
	\section{Proofs of Theorem \texorpdfstring{\ref{thm:1.1}}{1.1} and Theorem \texorpdfstring{\ref{thm:1.2}}{1.2}}\label{sect:4}
	%%%%%%%%%%%%%%%%%%%%%%%%%%%%%%%%%%%%%%%%%%%%%
	
	In this section, we first study the hypersurfaces of ${\rm Sol_0^4}$
	with constant principal curvatures and constant angle functions $c$ and $d$, as introduced in Section~\ref{sect:2.2}.
	
	\subsection{Proof of Theorem \texorpdfstring{\ref{thm:1.1}}{1.1}}~
	%%%%%%%%%%%%%%%%%%%%%%%%%%%%%%%%%%%%%%%%%%%%%
	
	Assume that $M$ is a hypersurface of ${\rm Sol_0^4}$ with constant principal curvatures
	and unit normal vector field $N=aE_1+bE_2+cE_3+dE_4$, where  $\{E_i\}_{i=1}^4$ is
	defined by~\eqref{eqn:2.2}.
	From $a^2+b^2+c^2+d^2=1$, it follows that $|d|\leq1$ and $|c|\leq1$. Moreover,
	up to changing the sign of the unit normal 
	and up to the action of $\mathbb{Z}_2$,
	we may always assume that $d\geq0$ and $c\geq0$.
	
	Since $c$ and $d$ are constants, it suffices to consider the following
	three cases:
	\vspace*{1ex}
	
	\begin{tabular}{l l}
		{\bf Case I} & $0\leq d<1$ and $c=0$ on $M$; \\
		{\bf Case II} & $0\leq d<1$ and $0<c\leq1$ on $M$; \\
		{\bf Case III} & $d=1$ and $c=0$ on $M$.
	\end{tabular}
	\vspace*{1ex}
	
	First, we focus on {\bf Case I}.
	In what follows, we prove that up to isometries of ${\rm Sol_0^4}$,
	$M$ is either an open part of $M_{1,r}$ for some $r>0$, or
	$M$ is an open part of $M_{2,r}$ for some $r\geq0$.
	Recall  that, in this case, we have $N=aE_1+bE_2+dE_4$, where $a^2+b^2+d^2=1$, $0\leq d<1$
	and $d$ is a constant.
	
	Consider the
	local orthonormal frame field $\{W_i\}_{i=1}^3$ defined by
	\begin{equation}\label{eqn:2.16}
		W_1=\frac{bE_1-a E_2}{\sqrt{1-d^2}}, \quad
		W_2=\frac{adE_1+bd E_2}{\sqrt{1-d^2}}-\sqrt{1-d^2}E_4, \quad
		W_3=E_3,
	\end{equation}
	where $\{E_i\}_{i=1}^4$ is defined by~\eqref{eqn:2.2}.
	In terms of the frame $\{W_i\}_{i=1}^3$, $J_1$, $J_2$ and $P$ are given by
	\begin{equation}\label{eqn:2.20}
		\begin{aligned}
			&J_1W_1=J_2W_1=dW_2+\sqrt{1-d^2}N, \quad J_1W_2=-dW_1+\sqrt{1-d^2}W_3,\\
			&J_2W_2=-dW_1-\sqrt{1-d^2}W_3, \quad J_1W_3=-J_2W_3=-\sqrt{1-d^2}W_2+dN,\\
			& PW_1=0, \quad PW_2=(1-d^2)W_2-d\sqrt{1-d^2}N, \quad PW_3=0.
		\end{aligned}
	\end{equation}

	Note that $M=\Omega_1\cup \Omega_2\cup \Omega_3$, where
	\begin{equation*}
		\begin{aligned}
			&\Omega_1=\{p\in M|\, b(p)\neq0, W_1(a)|_p\neq0\},\\
			&\Omega_2=\{p\in M|\, b(p)\neq0, W_1(a)|_p=0\},\\
			&\Omega_3=\{p\in M|\, b(p)=0\}.
		\end{aligned}
	\end{equation*}
	If $\Omega_1$ is non-empty, then $\Omega_1$ is an open subset of $M$.
	\begin{lemma}\label{lem:4.3}
		Assume that $\Omega_1$ is non-empty. 
		Then, on $\Omega_1$ it holds that $d\neq0$
		and
		\begin{equation}\label{eqn:2.18}
			W_1(a)=-\frac{b\sqrt{1-d^2}}{d}, \quad 
			W_1(b)=\frac{a\sqrt{1-d^2}}{d},\quad 
			W_2(a)=W_3(a)
			= W_2(b)=W_3(b)=0.
		\end{equation}
		Moreover, with respect to $\{W_i\}_{i=1}^3$,
		the shape operator is given by
		\begin{equation}\label{eqn:8.20}
			AW_1=\tfrac{1}{d}W_1, \quad AW_2=dW_2, \quad AW_3=-2dW_3.
		\end{equation}
		The connection coefficients with respect to $\{W_i\}_{i=1}^3$ are given by
		\begin{equation}\label{eqn:9.45}
			\begin{aligned}
				\nabla_{W_i}W_j&=0 \ ( i=1, 2, \ j=1, 2, 3), \quad \nabla_{W_3}W_1=0,\\
				\nabla_{W_3}W_2&=-2\sqrt{1-d^2}W_3, \quad
				\nabla_{W_3}W_3=2\sqrt{1-d^2}W_2.
			\end{aligned}
		\end{equation}
	\end{lemma}
	
	\begin{proof}
		On $\Omega_1$, using the fact that $a^2+b^2+d^2=1$, $b\neq0$ and $d$ is a constant,
		we have $W_i(b)=-\tfrac{a}{b}W_i(a)$ for $i=1,2,3$.
		Combining $W_1(b)=-\tfrac{a}{b}W_1(a)$ with \eqref{eqn:2.6} and the Weingarten formula,
		we obtain
		\begin{equation}\label{eqn:2.117}
			\begin{aligned}
				AW_1
				&=-\big(W_1(a)E_1+a\tilde{\nabla}_{W_1}E_1+W_1(b)E_2+b\tilde{\nabla}_{W_1}E_2
				+d\tilde{\nabla}_{W_1}E_4\big)\\
				&=\left(d-\frac{\sqrt{1-d^2}}{b} W_1(a) \right)W_1.
			\end{aligned}
		\end{equation}
		Similar calculations yield
		\begin{equation}\label{eqn:2.17}
			AW_2 = -\frac{\sqrt{1-d^2}}{b}W_2(a) W_1+dW_2,\quad
			AW_3 = -\frac{\sqrt{1-d^2}}{b}W_3(a) W_1-2dW_3.
		\end{equation}
		By symmetry of the shape operator, we obtain
		\begin{equation}\label{eqn:8.46}
			W_2(a)=W_3(a)=0,
		\end{equation}
		which implies that $W_2(b)=W_3(b)=0$.
		
		Using~\eqref{eqn:2.6}, \eqref{eqn:2.16}, \eqref{eqn:2.117},
		$W_1(b)=-\tfrac{a}{b}W_1(a)$ and the Gauss formula,
		we can calculate the connection to obtain
		\begin{equation*}
			\begin{aligned}
				\nabla_{W_1}W_1
				&=\frac{1}{\sqrt{1-d^2}}
				\left(W_1(b)E_1+b\tilde{\nabla}_{W_1}E_1-W_1(a)E_2
				-a\tilde{\nabla}_{W_1}E_2\right)
				-\left(d-\frac{\sqrt{1-d^2}}{b}W_1(a)\right)N\\
				&=-\left(\sqrt{1-d^2}+\frac{dW_1(a)}{b}\right)W_2.
			\end{aligned}
		\end{equation*}
		Similar calculations give the other connection coefficients:
		\begin{equation}\label{eqn:2.19}
			\begin{aligned}
				\nabla_{W_1}W_1&=-\left(\sqrt{1-d^2}+\frac{dW_1(a)}{b}\right)W_2, \quad
				\nabla_{W_1}W_2=\left(\sqrt{1-d^2}+\frac{dW_1(a)}{b}\right)W_1,\\
				\nabla_{W_1}W_3&=0, \quad \nabla_{W_2}W_1=0, \quad \nabla_{W_2}W_2=0, \quad \nabla_{W_2}W_3=0,\\
				\nabla_{W_3}W_1&=0, \quad \nabla_{W_3}W_2=-2\sqrt{1-d^2}W_3, \quad \nabla_{W_3}W_3=2\sqrt{1-d^2}W_2.
			\end{aligned}
		\end{equation}

		Taking $(X,Y,Z)=(W_1,W_2,W_1)$ in the Codazzi equation \eqref{eqn:2.15},
		with the use of \eqref{eqn:2.20}, \eqref{eqn:2.117}, \eqref{eqn:2.17},
		\eqref{eqn:8.46}, \eqref{eqn:2.19} and
		the assumption of constant principal curvatures, we obtain
		$$
		W_1(a)\Big(b(1-d^2)+d\sqrt{1-d^2}W_1(a)\Big)=0.
		$$
		Since $W_1(a)\neq0$ on $\Omega_1$, we have $b(1-d^2)+d\sqrt{1-d^2}W_1(a)=0$.
		So it follows from $b\neq0$ and $d<1$ that $d\neq0$. Therefore,
		$W_1(a)=-\tfrac{b\sqrt{1-d^2}}{d}$ and $W_1(b)=-\tfrac{a}{b}W_1(a)=\tfrac{a\sqrt{1-d^2}}{d}$.
		Substituting this into \eqref{eqn:2.117} and \eqref{eqn:2.19}, respectively,
		we obtain \eqref{eqn:8.20} and \eqref{eqn:9.45}.
	\end{proof}
	
	\begin{lemma}\label{lem:4.45}
		Assume that $\Omega_2$ is non-empty and contains an open subset $\Omega_{21}$ of $M$.
		Then,
		$a$ and $b$ are constants on $\Omega_{21}$.
		Moreover, with respect to $\{W_i\}_{i=1}^3$,
		the shape operator is given by
		\begin{equation}\label{eqn:8.21}
			AW_1=dW_1, \quad AW_2=dW_2, \quad AW_3=-2dW_3.
		\end{equation}
		The connection coefficients with respect to $\{W_i\}_{i=1}^3$ are given by
		\begin{equation}\label{eqn:9.46}
			\begin{aligned}
				\nabla_{W_1}W_1&=-\sqrt{1-d^2}W_2, \quad \nabla_{W_1}W_2=\sqrt{1-d^2}W_1, \quad \nabla_{W_1}W_3=0,\\
				\nabla_{W_2}W_1&=0, \quad \nabla_{W_2}W_2=0, \quad \nabla_{W_2}W_3=0,\\
				\nabla_{W_3}W_1&=0, \quad \nabla_{W_3}W_2=-2\sqrt{1-d^2}W_3, \quad \nabla_{W_3}W_3=2\sqrt{1-d^2}W_2.
			\end{aligned}
		\end{equation}
	\end{lemma}
	
	\begin{proof}
		Using a similar method as for the proof of Lemma~\ref{lem:4.3},
		we can show that the equations \eqref{eqn:2.117}--\eqref{eqn:2.19} also hold on $\Omega_{21}$.
		Combining $W_1(a)=0$ with~\eqref{eqn:8.46} on $\Omega_{21}$,
		we derive that $a$ and $b$ are constants.
		Substituting $W_1(a)=0$ into~\eqref{eqn:2.117} and~\eqref{eqn:2.19},
		we obtain~\eqref{eqn:8.21} and~\eqref{eqn:9.46}.
	\end{proof}
	
	\begin{lemma}\label{lem:4.46}
		Assume that $\Omega_3$ is non-empty and contains an open subset $\Omega_{31}$ of $M$.
		Then, $a$ is constant on $\Omega_{31}$. 
		Moreover, with respect to $\{W_i\}_{i=1}^3$,
		the shape operator is given by
		\begin{equation}\label{eqn:8.22}
			AW_1=dW_1, \quad AW_2=dW_2, \quad AW_3=-2dW_3.
		\end{equation}
		The connection coefficients with respect to $\{W_i\}_{i=1}^3$ are given by
		\begin{equation}\label{eqn:9.47}
			\begin{aligned}
				\nabla_{W_1}W_1&=-\sqrt{1-d^2}W_2, \quad \nabla_{W_1}W_2=\sqrt{1-d^2}W_1, \quad \nabla_{W_1}W_3=0,\\
				\nabla_{W_2}W_1&=0, \quad \nabla_{W_2}W_2=0, \quad \nabla_{W_2}W_3=0,\\
				\nabla_{W_3}W_1&=0, \quad \nabla_{W_3}W_2=-2\sqrt{1-d^2}W_3, \quad \nabla_{W_3}W_3=2\sqrt{1-d^2}W_2.
			\end{aligned}
		\end{equation}
	\end{lemma}
	
	\begin{proof}
		On $\Omega_{31}$, we have $a^2+d^2=1$.
		Therefore, since $d$ is a constant, $a$ is also a constant. Additionally,
		using the fact that $b=0$, together with the Gauss and Weingarten formulae and  equations~\eqref{eqn:2.6}   and~\eqref{eqn:2.16}, we obtain
		\begin{equation*}
			AW_1
			=-a\tilde{\nabla}_{W_1}E_1-d\tilde{\nabla}_{W_1}E_4
			=\frac{-ad}{\sqrt{1-d^2}}E_2=dW_1,
		\end{equation*}
		\begin{equation*}
			\nabla_{W_1}W_1=\tilde{\nabla}_{W_1}W_1-g(AW_1,W_1)N
			=-adE_1+(1-d^2)E_4=-\sqrt{1-d^2}W_2.
		\end{equation*}
		Similar calculations show that
		\eqref{eqn:8.22} and \eqref{eqn:9.47} hold.
	\end{proof}
	
	From Lemmas \ref{lem:4.3}--\ref{lem:4.46}, we derive the following result.
	\begin{lemma}\label{lem:4.47}
		In {\bf Case I}, one of the following three subcases occurs:
		\vspace*{1ex}
		
		\begin{tabular}{l l}
			{\bf Case I-(i)} & $\Omega_1$ is dense in $M$; \\
			{\bf Case I-(ii)} & $M=\Omega_2$; \\
			{\bf Case I-(iii)} & $M=\Omega_3$.
		\end{tabular}
	\end{lemma}
	
	\begin{proof}
		If $\Omega_1$ is non-empty, then both $\Omega_2$ and $\Omega_3$ do not
		contain any non-empty open subset of~$M$. In fact, from Lemma \ref{lem:4.3},
		the principal curvatures on $\Omega_1$ are $\tfrac{1}{d}$, $d$ and $-2d$.
		If $\Omega_2$ (resp. $\Omega_3$)
		contains a non-empty open subset $\Omega_{21}$ (resp. $\Omega_{31}$) of $M$, then from Lemma
		\ref{lem:4.45} (resp. Lemma \ref{lem:4.46}), the principal curvatures are $d$, $d$, $-2d$ on
		$\Omega_{21}$ (resp. $\Omega_{31}$).
		But this is impossible according to the constancy of the principal curvatures and $d\neq1$.
		Therefore, if $\Omega_1$ is non-empty, then $\Omega_1$ is dense in~$M$.
		
		If $\Omega_1$ is empty and $\Omega_2$ is non-empty, then $\Omega_2$ is an open subset of $M$.
		From Lemma \ref{lem:4.45}, we know that $b$ is a non-zero constant on $\Omega_2$.
		According to the fact that angle function $b$ is continuous, then
		$\Omega_3$ is empty, and $M=\Omega_2$.
		
		If both $\Omega_1$ and $\Omega_2$ are empty, then $M=\Omega_3$.
	\end{proof}
	
	In the following Lemmas (\ref{lem:4.4}--\ref{lem:4.6}) we focus on the subcases introduced in Lemma~\ref{lem:4.47}.
	\begin{lemma}\label{lem:4.4}
		Assume that {\bf Case I-(i)} occurs, then up to isometries of ${\rm Sol_0^4}$,
		$M$ is an open part of $M_{1,r}$ for some $r>0$.
	\end{lemma}
	
	\begin{proof}
		On $\Omega_{1}$, it can be directly checked that the Gauss and Codazzi equations are satisfied by using the
		local orthonormal frame field $\{W_i\}_{i=1}^3$ from~\eqref{eqn:2.16}, together with \eqref{eqn:2.20}, \eqref{eqn:8.20}
		and \eqref{eqn:9.45}. 
		
		By using \eqref{eqn:2.2} and \eqref{eqn:2.16}, we get
		\begin{equation}\label{eqn:2.23}
			W_1(t)=0, \quad W_2(t)=-\sqrt{1-d^2}, \quad W_3(t)=0,
		\end{equation}
		where $t$ is the fourth Cartesian coordinate function on $\mathbb{R}^4$.
		From \eqref{eqn:9.45} and $[W_i,W_j]=\nabla_{W_i}W_j-\nabla_{W_j}W_i$, we have
		\begin{equation}\label{eqn:2.223}
			[W_1,W_2]=0, \quad [W_1,W_3]=0, \quad [W_2,W_3]=2\sqrt{1-d^2}W_3.
		\end{equation}

		Now, we define a new frame field $\{X_i\}_{i=1}^3$ on $M$ as follows:
		\begin{equation}\label{eqn:2.34}
			X_1=W_1, \quad X_2=W_2, \quad X_3=e^{2t}W_3.
		\end{equation}
		By using \eqref{eqn:2.23}--\eqref{eqn:2.34},
		we can verify that $[X_i,X_j]=0$, $1\leq i<j\leq3$.
		Moreover, it follows from \eqref{eqn:2.18} and \eqref{eqn:2.34} that
		\begin{equation}
			\label{eqn:2.31}    
			X_1(a)=-\frac{b\sqrt{1-d^2}}{d}, \quad
			X_1(b)=\frac{a\sqrt{1-d^2}}{d}, \quad
			X_2(a)=X_3(a)=X_2(b)=X_3(b)=0.
		\end{equation}
		Since the vector fields $\{X_i\}_{i=1}^3$ commute and are linearly independent, we can locally identify $M$ with an open subset $\Omega$ of $\mathbb{R}^3$
		and express the hypersurface $M$ by an immersion
		\begin{equation*}
			\begin{aligned}
				\Phi:\Omega\subset\mathbb{R}^3 \longrightarrow {\rm Sol_0^4}:
				(u,v,w) \longmapsto (x(u,v,w),y(u,v,w),z(u,v,w),t(u,v,w)),
			\end{aligned}
		\end{equation*}
		such that
		\begin{equation}\label{eqn:2.36}
			\begin{aligned}
				d\Phi(\partial_u)
				&=(x_u,y_u,
				z_u,t_u)=X_1,\\
				d\Phi(\partial_v)
				&=(x_v,y_v,
				z_v,t_v)=X_2,\\
				d\Phi(\partial_w)
				&=(x_w,y_w,
				z_w,t_w)=X_3.
			\end{aligned}
		\end{equation}

		From \eqref{eqn:2.31} and \eqref{eqn:2.36}, we know
		that $a$ and $b$ depend only on $u$ and
		\begin{equation}\label{eqn:2.38}
			a_u=\frac{-b\sqrt{1-d^2}}{d}, \quad b_u=\frac{a\sqrt{1-d^2}}{d}.
		\end{equation}
		Solving the equation \eqref{eqn:2.38}, we obtain
		\begin{equation}\label{eqn:2.39}
			\begin{aligned}
				a(u)=c_1\cos\left(\tfrac{\sqrt{1-d^2}}{d}u\right)+c_2\sin\left(\tfrac{\sqrt{1-d^2}}{d}u\right), \quad
				b(u)=c_1\sin\left(\tfrac{\sqrt{1-d^2}}{d}u\right)-c_2\cos\left(\tfrac{\sqrt{1-d^2}}{d}u\right),
			\end{aligned}
		\end{equation}
		where $c_1^2+c_2^2=1-d^2$ and $c_1, c_2$ are constants.
		
		According to \eqref{eqn:2.2}, \eqref{eqn:2.16} and \eqref{eqn:2.34}, we have
		\begin{equation}\label{eqn:2.37}
			X_1= \left(\frac{be^t}{\sqrt{1-d^2}},-\frac{ae^t}{\sqrt{1-d^2}},0,0\right),\quad
			X_2= \left(\frac{ade^t}{\sqrt{1-d^2}},\frac{bde^t}{\sqrt{1-d^2}},0,-\sqrt{1-d^2}\right),\quad
			X_3= \left(0,0,1,0\right).
		\end{equation}
		The equations \eqref{eqn:2.36} and \eqref{eqn:2.37} show that
		$x$ and $y$ depend only on $(u,v)$, $z$ depends only on $w$,
		and $t$ depends only on $v$.
		It also holds that $t_v=-\sqrt{1-d^2}$, i.e.\ $t(v)=-\sqrt{1-d^2}v+c_3$, where $c_3$ is a constant.
		Hence, $v(t)=-\tfrac{t-c_3}{\sqrt{1-d^2}}$ and
		\begin{equation}\label{eqn:2.40}
			v_t=-\frac{1}{\sqrt{1-d^2}}.
		\end{equation}
		In what follows we use $t$ instead of $v$ as a local coordinate.
		
		From \eqref{eqn:2.36} and \eqref{eqn:2.39}--\eqref{eqn:2.40}, we have
		\begin{equation*}
			\begin{aligned}
				x_u
				&=\frac{e^t}{\sqrt{1-d^2}}
				\left(c_1\sin\left(\tfrac{\sqrt{1-d^2}}{d}u\right)-c_2\cos\left(\tfrac{\sqrt{1-d^2}}{d}u\right)\right),\\
				y_u
				&=\frac{-e^t}{\sqrt{1-d^2}}
				\left(c_1\cos\left(\tfrac{\sqrt{1-d^2}}{d}u\right)+c_2\sin\left(\tfrac{\sqrt{1-d^2}}{d}u\right)\right),\\
				x_t
				&=x_v v_t
				=\frac{-de^t}{1-d^2}
				\left(c_1\cos\left(\tfrac{\sqrt{1-d^2}}{d}u\right)+c_2\sin\left(\tfrac{\sqrt{1-d^2}}{d}u\right)\right),\\
				y_t
				&=y_v v_t
				=\frac{-de^t}{1-d^2}
				\left(c_1\sin\left(\tfrac{\sqrt{1-d^2}}{d}u\right)-c_2\cos\left(\tfrac{\sqrt{1-d^2}}{d}u\right)\right).
			\end{aligned}
		\end{equation*}
		By solving these equations, we obtain
		\begin{equation}\label{eqn:2.42}
			\begin{aligned}
				x(u,t) &= \frac{-de^t}{1-d^2}\left(c_1\cos\left(\tfrac{\sqrt{1-d^2}}{d}u\right)+c_2\sin\left(\tfrac{\sqrt{1-d^2}}{d}u\right)\right)+c_4,\\
				y(u,t) &= \frac{-de^t}{1-d^2}\left(c_1\sin\left(\tfrac{\sqrt{1-d^2}}{d}u\right)-c_2\cos\left(\tfrac{\sqrt{1-d^2}}{d}u\right)\right)+c_5,
			\end{aligned}
		\end{equation}
		where $c_4, c_5$ are constants.
		
		From \eqref{eqn:2.36} and \eqref{eqn:2.37}, we have $z_w=1$, i.e.\ 
		$z=w+c_6$, where $c_6$ is a constant.
		Now, by reparametrizing to use $z$ as a local coordinate, we have
		\begin{equation*}
			\begin{aligned}
				\Phi:\Omega\subset\mathbb{R}^3 \longrightarrow {\rm Sol_0^4}:
				(u,t,z)\longmapsto (x(u,t),y(u,t),z,t),
			\end{aligned}
		\end{equation*}
		where $x(u,t), y(u,t)$ are given by \eqref{eqn:2.42}.
		After a left translation by $(-c_4,-c_5,0,0)$,
		which is an isometry of ${\rm Sol_0^4}$, we obtain
		\begin{equation*}
			\begin{aligned}
				\Phi(u,t,z)
				=\Bigg(&\frac{-de^t}{1-d^2}\left(c_1\cos\left(\tfrac{\sqrt{1-d^2}}{d}u\right)
				+c_2\sin\left(\tfrac{\sqrt{1-d^2}}{d}u\right)\right), \\
				&\quad \frac{-de^t}{1-d^2}\left(c_1\sin\left(\tfrac{\sqrt{1-d^2}}{d}u\right)
				-c_2\cos\left(\tfrac{\sqrt{1-d^2}}{d}u\right)\right)
				,z,t\Bigg).
			\end{aligned}
		\end{equation*}

		Since $0<d<1$, we may assume that $d=\tanh r$ for some constant $r>0$.
		Let $L_1=
		\cosh r \left(\begin{smallmatrix}
			-c_1 & c_2 \\
			-c_2 & -c_1 \\
		\end{smallmatrix}\right)
		\in {\rm SO(2)}$.
		Then $L_1(\Phi(u,t,z))=(e^t\cos ({\rm csch}\, (r) u)\sinh r,
		e^t\sin ({\rm csch}\, (r) u)\sinh r,z,t)$.
		By taking the reparametrization $x_1={\rm csch}\, (r) u$, $x_2=z$
		and $x_3=t-\ln({\rm sech}\, r)$, we get
		$$
		L_1(\Phi(x_1,x_2,x_3))=(e^{x_3}\cos x_1\tanh r,e^{x_3}\sin x_1\tanh r,
		x_2,x_3+\ln({\rm sech}\, r)).
		$$
		This shows that $M$ is an open part of $M_{1,r}$ for some $r>0$.
	\end{proof}
	
	\begin{lemma}\label{lem:4.5}
		Assume that {\bf Case I-(ii)} occurs, then up to isometries of ${\rm Sol_0^4}$,
		$M$ is an open part of $M_{2,r}$ for some $r\geq0$.
	\end{lemma}
	
	\begin{proof}
		Recall that, in this case, $N=aE_1+bE_2+dE_4$ where $a^2+b^2+d^2=1$,
		$b\neq0$, $0\leq d<1$ and $a,b,d$ are constants.
		Applying Lemma \ref{lem:4.45},
		we can verify that all the Gauss and Codazzi equations are satisfied.
		
		Consider the rotation matrix $L_2= \tfrac{1}{\sqrt{1-d^2}}\left(
		\begin{smallmatrix}
			-b & a \\
			-a & -b \\
		\end{smallmatrix}
		\right)\in {\rm SO(2)}.$
		Then the unit normal vector field of $L_2(M)=\bar{M}$ is
		$\bar{N}=-\sqrt{1-d^2}E_2+dE_4$.
		An orthonormal frame field on $\bar{M}$ is then given by
		\begin{equation}\label{eqn:z2.23}
			W_1=-E_1, \quad W_2=-dE_2-\sqrt{1-d^2}E_4, \quad W_3=E_3.
		\end{equation}
		Note that the frame $\{W_i\}_{i=1}^3$ still satisfies Lemma \ref{lem:4.45} on $\bar{M}$.
		Then, by \eqref{eqn:9.46}, it holds that
		\begin{equation}\label{eqn:2.233}
			[W_1,W_2]=\sqrt{1-d^2}W_1, \quad [W_1,W_3]=0, \quad [W_2,W_3]=2\sqrt{1-d^2}W_3.
		\end{equation}
		From \eqref{eqn:2.2} and \eqref{eqn:z2.23}, we have
		\begin{equation}\label{eqn:x2.23}
			W_1(t)=0, \quad W_2(t)=-\sqrt{1-d^2}, \quad W_3(t)=0.
		\end{equation}

		Now, we define a new frame field $\{Y_i\}_{i=1}^3$ on $\bar M$ as follows:
		\begin{equation}\label{eqn:2.334}
			Y_1=e^{-t}W_1, \quad Y_2=W_2, \quad Y_3=e^{2t}W_3.
		\end{equation}
		By \eqref{eqn:2.233}--\eqref{eqn:2.334},
		we have $[Y_i,Y_j]=0$, for $1\leq i<j\leq3$.
		
		Since the vector fields $\{Y_i\}_{i=1}^3$ commute and are linearly independent, we can locally identify $\bar{M}$ with an open subset $\Omega$ of $\mathbb{R}^3$
		and express the hypersurface $\bar{M}$ by an immersion
		\begin{equation*}
			\Phi:\Omega\subset\mathbb{R}^3 \longrightarrow {\rm Sol_0^4}:
			(u,v,w) \longmapsto (x(u,v,w),y(u,v,w),z(u,v,w),t(u,v,w)),
		\end{equation*}
		such that
		\begin{equation}\label{eqn:2.47}
			\begin{aligned}
				d\Phi(\partial_u)
				&=(x_u,y_u,
				z_u,t_u)
				=Y_1,\\
				d\Phi(\partial_v)
				&=(x_v,y_v,
				z_v,t_v)
				=Y_2,\\
				d\Phi(\partial_w)
				&=(x_w,y_w,
				z_w,t_w)
				=Y_3.
			\end{aligned}
		\end{equation}
		From \eqref{eqn:2.2}, \eqref{eqn:z2.23} and \eqref{eqn:2.334}, we also have
		\begin{equation}\label{eqn:2.48}
			Y_1 =(-1,0,0,0),\quad
			Y_2 =\left(0,-de^t,0,-\sqrt{1-d^2} \right),\quad
			Y_3 =(0,0,1,0).
		\end{equation}
		The equations \eqref{eqn:2.47} and \eqref{eqn:2.48} show that
		$x$ depends only on $u$, $y$ depends only on $v$,
		$z$ depends only on $w$ and $t$ depends only on $v$.
		Moreover,~\eqref{eqn:2.40} also holds in this case, which allows us to use $t$ instead of $v$ as a local coordinate.
		
		From \eqref{eqn:2.47} and \eqref{eqn:2.48}, we have
		$x_u=-1$ and $z_w=1$, i.e.\
		$x=-u+c_7$ and $z=w+c_{8}$
		where $c_7, c_8$ are constants.
		From \eqref{eqn:2.40}, \eqref{eqn:2.47} and \eqref{eqn:2.48},
		we get $y_t=\tfrac{de^t}{\sqrt{1-d^2}}$.
		Solving this equation, we obtain
		$y=\tfrac{de^t}{\sqrt{1-d^2}}+c_{9}$
		where $c_{9}$ is a constant.
		
		Using $x$ and $z$ as the new local coordinates, we obtain
		$\Phi(x,t,z) = (x,\tfrac{de^t}{\sqrt{1-d^2}}+c_{9},z,t)$.
		After a left translation by $(0,-c_{9},0,0)$,
		which is an isometry of ${\rm Sol_0^4}$, we obtain
		\begin{equation}\label{eqn:9.70}
			\Phi(x,t,z)=\left(x,\frac{de^t}{\sqrt{1-d^2}},z,t\right).
		\end{equation}
		
		Since $0\leq d<1$, we may assume that $d=\tanh r$ for some constant $r\geq0$.
		By taking the reparametrization $x_1=x$, $x_2=z$ and
		$x_3=t-\ln({\rm sech}\, r)$, we obtain
		$$
		\Phi(x_1,x_2,x_3)=\left(x_1,e^{x_3}\tanh r,x_2,x_3+\ln ({\rm sech}\, r)\right).
		$$
		This shows that, up to isometries of ${\rm Sol_0^4}$,
		$M$ is an open part of $M_{2,r}$ for some $r\geq0$.
	\end{proof}
	
	\begin{lemma}\label{lem:4.6}
		Assume that {\bf Case I-(iii)} occurs, then up to isometries of ${\rm Sol_0^4}$,
		$M$ is an open part of $M_{2,r}$ for some $r\geq0$.
	\end{lemma}
	
	\begin{proof}
		Recall that, in this case, we have $N=aE_1+dE_4$, where $a^2+d^2=1$,
		$0\leq d<1$ and $a,d$ are constants.
		If necessary, up to the action of
		$\left(
		\begin{smallmatrix}
			-1 & 0 \\
			0 & 1 \\
		\end{smallmatrix}
		\right)\in {\rm O(2)}$, we can always assume that $a=-\sqrt{1-d^2}$ on $M$.
		
		Let $L_3=\left(
		\begin{smallmatrix}
			0 & 1 \\
			1 & 0 \\
		\end{smallmatrix}
		\right)\in {\rm O(2)}$.
		Then the hypersurface $L_3(M)$ has constant principal curvatures and  unit normal vector field $-\sqrt{1-d^2}E_2+dE_4$.
		By applying Lemmas \ref{lem:4.47} and \ref{lem:4.5} to $L_3(M)$,
		we know that, up to isometries of ${\rm Sol_0^4}$,
		$M$ is an open part of $M_{2,r}$ for some $r\geq0$.
	\end{proof}

	We now focus on {\bf Case II}.
	In what follows, we prove that up to isometries of ${\rm Sol_0^4}$,
	$M$ is an open part of $M_{3,r}$ for some $r\geq0$.
	Recall that, in this case, we have $N=aE_1+bE_2+cE_3+dE_4$, where
	$a^2+b^2+c^2+d^2=1$, $0\leq d<1$, $0<c\leq 1$ and $c,d$ are constants.
	\begin{lemma}
		\label{lemma:case2-ab-zero}
		Assume that {\bf Case II} occurs, then $a$ and $b$ are identically zero on $M$.
	\end{lemma}
	\begin{proof}
		Suppose, on the contrary, that $a\neq0$ in an open subset $\Omega$ of $M$.
		Then, we have $T_i(a)=-\tfrac{b}{a}T_i(b)$ for $i=1,2,3$, where $\{T_i\}_{i=1}^3$
		is an orthonormal frame field defined
		by \eqref{eqn:8.1}.
		Combining these with~\eqref{eqn:2.6},
		the Weingarten formula and the assumption of $c, d$ being constants, we obtain
		\begin{equation*}
			\begin{aligned}
				AT_1
				&=\left(a^2d+d(b^2-2(c^2+d^2))+\tfrac{(a^2+b^2)T_1(b)}{a}\right)T_1 \\
				& \ \ \ +\tfrac{(bc+ad)(3ad+T_1(b))}{a}T_2-\tfrac{(ac-bd)(3ad+T_1(b))}{a}T_3,\\
				AT_2
				&=\left(\tfrac{(a^2+b^2)T_2(b)}{a}+3a(c^2+d^2)\right)T_1 \\
				& \ \ \ +\left(-3abc-2a^2d+\tfrac{bcT_2(b)}{a}+d(b^2+c^2+d^2+T_2(b))\right)T_2 \\
				& \ \ \ +\tfrac{(ac-bd)(3a^2-T_2(b))}{a}T_3,\\
				AT_3
				&=\left(3b(c^2+d^2)+\tfrac{(a^2+b^2)T_3(b)}{a}\right)T_1
				-\tfrac{(bc+ad)(3ab-T_3(b))}{a}T_2\\
				& \ \ \ +\left(3abc+a^2d-2b^2d+d(c^2+d^2)-cT_3(b)+\tfrac{bdT_3(b)}{a}\right)T_3. \\
			\end{aligned}
		\end{equation*}
		By symmetry of the shape operator we obtain
		\begin{align}
			d(3bc+T_1(b))-a(3c^2+T_2(b))+\tfrac{b(cT_1(b)-bT_2(b))}{a}&=0,\label{eqn:4.3}\\
			-c(3bc+T_1(b))-a(3cd+T_3(b))+\tfrac{b(dT_1(b)-bT_3(b))}{a}&=0,~\label{eqn:4.4}\\
			3a^3c+3ab^2c+(bd-ac)T_2(b)-(bc+ad)T_3(b)&=0.~\label{eqn:4.5}
		\end{align}
		From \eqref{eqn:4.3} and \eqref{eqn:4.4}, we get
		\begin{equation*}
			T_2(b) = \frac{-3a^2c^2+bcT_1(b)+ad(3bc+T_1(b))}{a^2+b^2},\quad
			T_3(b) = \frac{-3ac(bc+ad)+(bd-ac)T_1(b)}{a^2+b^2}.
		\end{equation*}
		Substituting these into \eqref{eqn:4.5}, we obtain $3ac=0$ on $\Omega$,
		which contradicts with the assumption that $a\neq0$ and $c > 0$.
		Thus $a=0$ on $M$.
		With a similar argument, we can also prove that $b=0$ on~$M$.
	\end{proof}
	
	\begin{lemma}
		\label{lemma:case2-immersion}
		Assume that {\bf Case II} occurs, then up to isometries of ${\rm Sol_0^4}$,
		$M$ is an open part of $M_{3,r}$ for some $r\geq0$.
	\end{lemma}
	\begin{proof}
		By Lemma~\ref{lemma:case2-ab-zero}, we know that $N=cE_3+dE_4$, where $c^2+d^2=1$, $0<c\leq 1$,
		$0\leq d<1$ and $c$ and $d$ are constants. 
		Then the following is an orthonormal frame field on $M$:
		\begin{equation}\label{eqn:2.70}
			\overline{W}_1=E_1, \quad \overline{W}_2=E_2, \quad \overline{W}_3=dE_3-cE_4.
		\end{equation}
		From \eqref{eqn:2.2}, we have
		\begin{equation}\label{eqn:2.z70}
			\overline{W}_1(t)=0, \quad \overline{W}_2(t)=0, \quad \overline{W}_3(t)=-c.
		\end{equation}
		By using \eqref{eqn:2.6}, \eqref{eqn:2.70}, the Gauss and Weingarten formulae,
		we obtain
		\begin{equation}\label{eqn:2.71}
			\begin{alignedat}{3}
				&\nabla_{\overline{W}_1}\overline{W}_1=-c\overline{W}_3, \quad
				&&\nabla_{\overline{W}_1}\overline{W}_2=0, \quad
				&&\nabla_{\overline{W}_1}\overline{W}_3=c\overline{W}_1,\\
				&\nabla_{\overline{W}_2}\overline{W}_1=0, \quad
				&&\nabla_{\overline{W}_2}\overline{W}_2=-c\overline{W}_3, \quad
				&&\nabla_{\overline{W}_2}\overline{W}_3=c\overline{W}_2,\\
				&\nabla_{\overline{W}_3}\overline{W}_1=0, \quad
				&&\nabla_{\overline{W}_3}\overline{W}_2=0, \quad
				&&\nabla_{\overline{W}_3}\overline{W}_3=0,\\
				&A\overline{W}_1=d\overline{W}_1, \quad 
				&&A\overline{W}_2=d\overline{W}_2, \quad
				&&A\overline{W}_3=-2d\overline{W}_3.
			\end{alignedat}
		\end{equation}
		Then it can be checked that all the Gauss and Codazzi equations are satisfied.
		From \eqref{eqn:2.71}, we get
		\begin{equation}\label{eqn:2.l23}
			[\overline{W}_1,\overline{W}_2]=0, \quad [\overline{W}_1,\overline{W}_3]=c\overline{W}_1, \quad [\overline{W}_2,\overline{W}_3]=c\overline{W}_2.
		\end{equation}

		Now, we define a new frame field $\{Z_i\}_{i=1}^3$ on $M$ as follows:
		\begin{equation}\label{eqn:2.3341}
			Z_1=e^{-t}\overline{W}_1, \quad Z_2=e^{-t}\overline{W}_2, \quad Z_3=\overline{W}_3.
		\end{equation}
		By using \eqref{eqn:2.z70}, \eqref{eqn:2.l23} and \eqref{eqn:2.3341},
		we obtain $[Z_i,Z_j]=0$, $1\leq i<j\leq3$.
		
		Since the vector fields $\{Z_i\}_{i=1}^3$ commute and are linearly independent, we may locally identify $M$ with an open subset $\Omega$ of $\mathbb{R}^3$
		and express the hypersurface $M$ by an immersion
		\begin{equation*}
			\begin{aligned}
				\Phi:\Omega\subset\mathbb{R}^3 &\longrightarrow {\rm Sol_0^4},\\
				(u,v,w) &\longmapsto (x(u,v,w),y(u,v,w),z(u,v,w),t(u,v,w)),
			\end{aligned}
		\end{equation*}
		such that
		\begin{equation}\label{eqn:2.76}
			\begin{aligned}
				d\Phi(\partial_u)
				&=(x_u,y_u,
				z_u,t_u)=Z_1,\\
				d\Phi(\partial_v)
				&=(x_v,y_v,
				z_v,t_v)=Z_2,\\
				d\Phi(\partial_w)
				&=(x_w,y_w,
				z_w,t_w)=Z_3.
			\end{aligned}
		\end{equation}
		From \eqref{eqn:2.2}, \eqref{eqn:2.70} and \eqref{eqn:2.3341}, we also have
		\begin{equation}\label{eqn:2.77}
			Z_1=(1,0,0,0), \quad Z_2=(0,1,0,0), \quad Z_3=\left(0,0,de^{-2t},-c\right).
		\end{equation}
		The equations \eqref{eqn:2.76} and \eqref{eqn:2.77} show that
		$x$ depends only on $u$, $y$ depends only on $v$, $z$ and $t$ depend only on $w$.
		
		From~\eqref{eqn:2.76} and \eqref{eqn:2.77}, it follows that
		$x=u+c_{10}$, $y=v+c_{11}$,
		where $c_{10}, c_{11}$ are constants.
		By \eqref{eqn:2.76} and \eqref{eqn:2.77}, we have $t_w=-c$, i.e.\ $t=-cw+c_{12}$, where $c_{12}$ is a constant.
		Hence,
		$w=-\tfrac{t-c_{12}}{c}$ and
		\begin{equation}\label{eqn:2.z40}
			w_t=-\frac{1}{c}.
		\end{equation}
		In what follows, we use $t$ instead of $w$ as a local coordinate.
		Again, by using \eqref{eqn:2.76}, \eqref{eqn:2.77} and \eqref{eqn:2.z40}, we get
		$z_t=-\tfrac{d}{c}e^{-2t}$. Hence,
		$z=\tfrac{d}{2c}e^{-2t}+c_{13}$, where $c_{13}$ is a constant.	
		Using $x$ and $y$ instead of $u$ and $v$ as local coordinates,
		we obtain
		$\Phi(x,y,t)= \left(x,y,\tfrac{d}{2c}e^{-2t}+c_{13},t\right)$.
		After a left translation by $(0,0,-c_{13},0)$,
		we obtain
		$
		\Phi(x,y,t)=\left(x,y,\tfrac{d}{2c}e^{-2t},t\right).
		$
		
		Since $0\leq d<1$, we may assume that $d=\tanh (2r)$ for some $r\geq0$.
		By taking the reparametrization $x_1=x$, $x_2=y$ and
		$x_3=t-\tfrac{1}{2}\ln(\cosh(2r))$, we get
		$$
		\Phi(x_1,x_2,x_3)=(x_1,x_2,\tfrac{1}{2}e^{-2x_3}\tanh(2r),x_3+\tfrac{1}{2}\ln(\cosh(2r))).
		$$
		This shows that $M$ is an open part of $M_{3,r}$ for some $r\geq0$.
	\end{proof}
	
	Finally, we consider {\bf Case III}. In this case, the unit normal
	vector field is $N=E_4$, and $TM=\text{span}\{E_1,E_2,E_3\}$,
	where $\{E_i\}_{i=1}^4$ is defined by \eqref{eqn:2.2}. By using the equations \eqref{eqn:2.2} and \eqref{eqn:2.3}, we can derive that $M$ is an open part of
	$\{(x,y,z,r)\in {\rm Sol_0^4}\mid x,y,z\in\mathbb{R}\}$
	for some $r\in\mathbb{R}$, which is congruent to $M_{4,0}$.
	
	To complete the proof, we refer to Propositions~\ref{prop:M1r}--\ref{prop:M4r}, where we have shown that the hypersurfaces $M_{1,r}$, $M_{2,r}$, $M_{3,r}$ and $M_{4,0}$ have constant principal curvatures and satisfy the properties mentioned in the statement of the theorem.
	This concludes the proof of Theorem~\ref{thm:1.1}.
	
	\subsection{Proof of Theorem \texorpdfstring{\ref{thm:1.2}}{1.2}}~
	%%%%%%%%%%%%%%%%%%%%%%%%%%%%%%%%%%%%%%%%%%%%%
	
	Let $M$ be a homogeneous hypersurface of ${\rm Sol_0^4}$, 
	then $M$ is an orbit of some closed subgroup $\widetilde{G}\subset{\rm Sol_0^4}\rtimes({\rm O(2)}\times\mathbb{Z}_2)$. 
	Let $N$ be a unit normal vector field of $M$ and fix a point $p_{0} \in M$.
	From the assumption that $M$ is homogeneous, for any $p \in M$,
	there exists an isometry $\phi\in {\rm Sol_0^4}\rtimes({\rm O(2)}\times\mathbb{Z}_2)$
	such that $\phi\left(M\right)=M$ and $\phi(p_{0})=p$.
	From $\phi(M)=M$ it follows that $N_{p}=\pm d \phi_{p_{0}}(N_{p_{0}})$, and from~\eqref{eqn:2.1},~\eqref{eqn:2.2} and~\eqref{eqn:2.100} it follows that
	$d\phi_{p_{0}} (E_3|_{p_{0}})=\pm E_3|_{p}$ and
	$d\phi_{p_{0}} (E_4|_{p_{0}})=E_4|_{p}$.
	Hence, we have
	$$
	\begin{aligned}
		c(p)&=g(N_{p},E_3|_{p})=g(\pm d \phi_{p_{0}}(N_{p_{0}}),\pm d\phi_{p_{0}} (E_3|_{p_{0}}))=\pm c(p_{0}),\\
		d(p)&=g(N_{p},E_4|_{p})=g(\pm d \phi_{p_{0}}(N_{p_{0}}),d\phi_{p_{0}} (E_4|_{p_{0}}))=\pm d(p_{0}).
	\end{aligned}
	$$
	Because $M$ is connected, these equations imply that
	the angle functions $c$ and $d$ are constant on~$M$.
	
	Since $M$ is a homogeneous hypersurface, it has constant principal curvatures. Then, from
	Theorem \ref{thm:1.1},
	we know that, up to isometries of ${\rm Sol_0^4}$,
	$M$ is either $M_{1,r}$ for some $r>0$, or $M$ is $M_{2,r}$ for some $r\geq0$,
	or $M$ is $M_{3,r}$ for some $r\geq0$, or $M$ is $M_{4,0}$.
	Conversely, we have already proved in Propositions \ref{prop:M1r}--\ref{prop:M4r}
	that these hypersurfaces are homogeneous.
	This concludes the proof of Theorem \ref{thm:1.2}.

		\begin{fund}
		M. D'haene was supported by Methusalem grant METH/21/03-long term structural funding of the Flemish Government and FWO (Research Foundation Flanders) grant K217724N. 
		G. Wei was supported by NSFC Grant No. 12171164 and Guangdong Natural Science Foundation Grant No. 2023A1515010510. 
		Z. Yao was supported by NSFC Grant No. 12401061. 
		X. Zhang was supported by China Postdoctoral Science Foundation Grant No. 2025M773115. 
	\end{fund}
	
	\begin{acknow}
		The authors would like to thank Professor Luc Vrancken, Professor Zejun Hu and Dr.\ Mateo Anarella for their helpful         conversations on this work. 
	\end{acknow}
	
	\vskip 0.1cm
	\noindent\textbf{Data availibility}
	Data sharing not applicable to this article as no datasets were generated or analysed during the current study.
	
	\vskip 0.1cm
	\noindent\textbf{Conflict of interest} On behalf of all authors, the corresponding author states that there is no conflict of interest.
	
	%%%%%%%%%%%%%%%%%%%%%%%%%%%%%%%%%%%%%%%%%%%%%%%%%%%%%%%%%%%%%%%%%%%%

\end{document}